\documentclass[a4paper,11pt,reqno]{amsart}
\usepackage{enumerate}
\usepackage{amsfonts}
\usepackage{color}
\usepackage{amsmath}
\usepackage{amssymb}
\usepackage[latin1]{inputenc}

\newcommand{\R}{{\mathbb R}}
\newcommand{\N}{{\mathbb N}}

\newcommand{\+}{\hspace{0.1em}}
\newcommand{\Vk}{V\!\wedge k}
\newcommand{\scpr}[2]{\bigl\langle #1,#2\bigr\rangle}

\newcommand{\calH}{\mathcal{H}}
\newcommand{\calA}{\mathcal{A}}
\newcommand{\calB}{\mathcal{B}}

\newcommand{\eps}{{\varepsilon}}

\DeclareSymbolFont{bbold}{U}{bbold}{m}{n}
\DeclareSymbolFontAlphabet{\mathbbold}{bbold}
\newcommand{\ind}{\mathbbold{1}}
\newcommand{\from}{\colon}
\newcommand{\set}[2]{\bigl\{#1\bigm|#2\bigr\}}

\renewcommand\le{\leqslant}
\renewcommand\ge{\geqslant}
\renewcommand{\b}{{\rm b}}
\renewcommand{\c}{{\rm c}}

\DeclareMathOperator{\dist}{dist}
\DeclareMathOperator{\spt}{spt}
\newcommand{\support}{\spt}

\newtheorem{lemma}{Lemma}[section]
\newtheorem{theorem}[lemma]{Theorem}
\newtheorem{coro}[lemma]{Corollary}
\newtheorem{prop}[lemma]{Proposition}

\theoremstyle{definition}
\newtheorem{definition}[lemma]{Definition}
\newtheorem{example}[lemma]{Example}

\newenvironment{Proof}{{\noindent\emph Proof}\;}
{\hfill$\square$\par\medskip} \newlength\headseptemp

\newcommand{\Hmm}[1]{\leavevmode{\marginpar{\tiny%
$\hbox to 0mm{\hspace*{-0.5mm}$\leftarrow$\hss}%
\vcenter{\vrule depth 0.1mm height 0.1mm width \the\marginparwidth}%
\hbox to 0mm{\hss$\rightarrow$\hspace*{-0.5mm}}$\\\relax\raggedright #1}}}
\begin{document}
\title[Large time behaviour of heat kernels]{Note on basic features of large time behaviour of heat kernels}
\author[]{Matthias Keller$^1$}
\author[]{Daniel Lenz$^2$}
\author[]{Hendrik Vogt$^3$}
\author[]{Rados{\l}aw Wojciechowski$^4$}

\address{$^1$ Mathematisches Institut, Friedrich Schiller Universit\"at Jena,
  D-03477 Jena, Germany, m.keller@uni-jena.de}
\address{ $^2$ Mathematisches Institut, Friedrich Schiller Universit\"at Jena,
  D-03477 Jena, Germany, daniel.lenz@uni-jena.de,
  URL: http://www.analysis-lenz.uni-jena.de/ }
\address{$^3$ Institut f\"ur Mathematik, Technische Universit\"at Hamburg-Harburg,\break
  D-21073 Hamburg, Germany, hendrik.vo\rlap{\textcolor{white}{hugo@egon}}gt@tu-\rlap{\textcolor{white}{darmstadt}}harburg.de}
\address{$^4$ York College of the City University of New York, Jamaica, NY 11451, USA, rwojciechowski@york.cuny.edu
}

\begin{abstract}
Large time behaviour of heat semigroups (and, more generally, of positive selfadjoint semigroups) is studied.
Convergence of the semigroup to the ground state and of averaged logarithms of kernels to the ground state energy is shown in the general framework of positivity improving selfadjoint semigroups. This framework encompasses all irreducible semigroups coming from Dirichlet forms as well as suitable perturbations thereof. It includes, in particular, Laplacians on connected manifolds, metric graphs and discrete graphs.
\end{abstract}
\date{\today} %
\maketitle
\begin{center}
\emph{}
\end{center}



\abovedisplayshortskip=2pt plus 3pt
\belowdisplayshortskip=9pt plus 3pt minus 5pt
\newcommand\smallbds{\vskip-1\lastskip\vskip7pt plus3pt minus3pt\noindent}

\section*{Introduction}

The study of the heat equation has a long history. There is a vast amount of literature devoted to heat kernel estimates under various geometric assumptions.  Here, we want to investigate two basic issues concerning long term behaviour of the heat semigroup, which turn out to be rather independent of the underlying geometry.  These issues are:

\begin{itemize}

\item Convergence of the semigroup to the ground state.

\item Convergence of averaged logarithms of the kernels to the infimum of the spectrum.

\end{itemize}

In differential geometry these topics have been studied both for compact and non-compact manifolds. In the compact case the results are well known. In the general case, the first issue is settled by a result of Chavel and Karp \cite{CK} (see Simon \cite{Sim} for a simplification as well), and the second is known as theorem of Li after \cite{Li}, where a statement can be found. (The paper itself does not seem to contain a proof but rather provides much stronger estimates under additional curvature assumptions). Corresponding results on heat equations with an elliptic generator can also be found in Pinchover's work \cite{Pin1,Pin2}.

In probability theory these points are well known for (continuous time) Markov chains on a finite state space due to the Perron-Frobenius theorem. We are not aware of a treatment for general Markov chains on an infinite state space.

\smallskip

Here, we present a new approach to these two issues in the general framework of arbitrary positivity improving selfadjoint semigroups. This framework covers a large array of examples, among them Laplacians on manifolds, metric graphs and discrete graphs. In particular, we recover the mentioned results of \cite{CK,Li,Sim} and provide results for (continuous time) Markov chains on infinite state spaces.  A short way of phrasing our result would be that existence of kernels alone already implies the above long term behaviour (irrespective of the underlying setting or geometry).

\smallskip

Let us emphasize that we do not require at all the existence of a ground state. In fact, in situations with a ground state the above long term behaviour can be rather easily inferred. In this sense, a major achievement of our approach is to be applicable irrespective of the existence of a ground state.

\smallskip

The crucial new insight for our unified treatment in the mentioned generality
is that any positive function completely controls the bottom of the spectrum of a positivity improving selfadjoint semigroup.
The proof of this fact is based on
the simple observation that, for a strictly positive function $h$ on the measure space $(X,m)$, the set
  $$\set{ u \in L^2 (X,m) }{ 0 \le u \le h }$$
is total in $L^2 (X,m)$
(see Theorem~\ref{positivitypreserving} and its proof for details).

\smallskip

As a consequence of our investigations and earlier results of \cite{voi86,voi88} we can substantially generalize a result of Cabr\'{e} and Martel \cite{CM} on existence of positive weak solutions of the heat equation with a strongly negative potential. In the same spirit we can generalize a recent result on Kolmogorov operators due to Goldstein/Goldstein/Rhandi \cite{GGR}.

We also note in passing that an application of our results to graphs positively answers a question raised by Weber in \cite{Web2}.

 \smallskip

Unlike other basic results of semigroup theory, our results crucially depend on the selfadjointness of the underlying semigroup. In fact, they become wrong for general positivity improving semigroups as we show by an example below (see, however, \cite{Pin1,Pin2} for a treatment of certain non-selfadjoint semigroups in strongly local situations).

\smallskip

\smallskip

We develop our general results in three steps, first discussing general semigroups of selfadjoint operators in Section~\ref{Framework}, then turning to positivity improving semigroups in Section~\ref{Positivity} and finally turning to positivity improving semigroups with kernels in Section~\ref{Semigroups}. An application to negative perturbations of positivity improving semigroups is presented in Section~\ref{Application}. This provides the above mentioned generalization of the result of \cite{CM}. In Section~\ref{Examples} we discuss examples viz Laplacians on manifolds, metric graphs and discrete graphs. This section contains the answer to the mentioned question of \cite{Web2}.
 The (counter)example proving that selfadjointness is needed for our considerations is given at the end of Section~\ref{Positivity}.

The topic of the paper concerns the intersection of various subjects. Thus, not all readers may be familiar with the general theory of positivity improving semigroups. For this reason we include an appendix gathering various basic pieces of this theory.

\section{Framework and basic result}\label{Framework}

In this section we introduce the framework used throughout the paper. We then present the basic result on convergence concerning the two issues discussed in the introduction.

\bigskip

We consider a selfadjoint operator $L$ in a Hilbert space $\calH$. The inner product on $\calH$ is denoted by $\langle \cdot,\cdot\rangle$. The operator $L$ is assumed to be bounded below. Hence, the operators $e^{-tL}$, $t\ge0$, form a semigroup of bounded operators. The behaviour of this semigroup for large $t$ is the focus of attention in the present work.
The infimum of the spectrum of $L$ is denoted by $E_0 = E_0 (L)$. The projection onto the eigenspace associated to $E_0$ is denoted by $P$, i.e.,
 $$ P := \ind_{\{E_0\}} (L),$$
 where $\ind_{X}$ denotes the indicator function of the set $X$.
 Note that $P =0$ if $E_0$ is not an eigenvalue.  In any case, we speak about $E_0$ as the \textit{ground state energy}.  The spectral measure of an element $f\in \calH$ with respect to $L$ is denoted by $\rho_f$. It is a finite measure on $[E_0,\infty)$ with the characteristic property that
 $$\langle f, e^{-t L} f\rangle = \int_{[E_0,\infty)} e^{- t s} \, d\rho_f (s)$$
for all $t\ge 0$.

The topological support of the measure $\rho_f$ is given by
 $$\support (\rho_f) = \set{E\in \R}{ \rho_f (E - \delta, E+ \delta)>0\;\:\mbox{for all $\delta >0$}}.$$

\bigskip

The following is a rather immediate consequence of the spectral theorem (and part~(a) is already contained in the considerations of \cite{Sim}).

\begin{theorem}\label{spectralversion}
Let $L$ be a selfadjoint operator in the Hilbert space $\calH$ with infimum of the spectum $E_0 > -\infty$ and let $P = \ind_{\{E_0\}}(L)$. Then the following holds:

\textup{(a)} The operators $e^{t E_0} e^{- t L}$ converge strongly to $P$ for $t\to \infty$, i.e.,
 $$ e^{t E_0} e^{- t L} f \to P f \qquad (t\to\infty) $$
for all $f\in \calH$.

\smallskip

\textup{(b)} For any $f\in \calH$ with $f\ne 0$ the equality
 $$\lim_{t\to \infty} \frac{ \log \langle f, e^{- tL} f\rangle}{t}= - \inf \support (\rho_f )$$
holds.
\end{theorem}
\begin{proof}
(a) The spectral theorem gives
 $$ \| (e^{tE_0} e^{-t L} - P) f\|^2 = \int_{[E_0,\infty)} |e^{t E_0} e^{- ts} - \ind_{\{E_0\}} (s)|^2 \, d\rho_f (s).$$
Obviously, the integrand is bounded by $1$ and tends to zero everywhere. Hence, the Lebesgue convergence theorem gives (a).

\smallskip

(b) Let $E_f := \inf \support (\rho_f )$ and let $t,\delta>0$. The spectral theorem easily yields
 $$e^{- (E_f + \delta) t} \| \ind_{[E_f, E_f + \delta]} (L) f\|^2 \le \int_{[E_f,\infty)} e^{- ts} \, d\rho_f (s) \le e^{- E_f t} \|f\|^2.$$
We can take logarithms in this inequality as
 $$ \ind_{[E_f, E_f + \delta]} (L) f\ne 0$$
by the definition of $\support \rho_f$, obtaining
 $$ -(E_f+\delta)t + 2\ln \| \ind_{[E_f, E_f + \delta]} (L) f\|
    \le \ln \langle f,e^{-tL}f \rangle \le -E_f t + 2\ln\|f\|. $$
After division by $t$, the desired statement then follows by considering first the limit $t\to \infty$ and then $\delta\to 0$.
\end{proof}

For later use we note that any selfadjoint $L$ which is bounded below by a constant $C\in\R$ comes with a closed form $Q$ defined by
$$ Q (f,g) :=\langle (L -C)^{1/2} f, (L - C)^{1/2} g\rangle + C \langle f, g\rangle$$
for $f,g\in D(Q) := D ( (L -C)^{1/2})$. Note that this definition does not depend on the actual choice of $C$, provided that $L\ge C$.

\section{Positivity improving semigroups}\label{Positivity}

In this section we specialize the framework of the last section to positivity preserving selfadjoint semigroups. This will allow us to strengthen the result on convergence to the ground state energy of the previous section.

We note already here the following, with notation to be explained below:
Our considerations hinge on strict positivity of expressions of the form $e^{-t L} f$.  In the case that $e^{-t L}$ is only positivity \emph{preserving}, we will therefore require strict positivity of $f$. For positivity \emph{improving} $e^{-t L}$ it suffices to assume the corresponding $f$ to be non-negative (and non-trivial).

\bigskip

Let $\calH = L^2 (X,m)$, where $X$ is a measure space with $\sigma$-algebra $\calB$ and a $\sigma$-finite measure $m$. A function $f$ on $X$ is called \textit{positive} if
 $$\mbox{ $\langle f, f\rangle >0$ and $ f(x) \ge 0 $ for $m$-almost every $x\in X$}.$$
 A function $f$ on $X$ is called \textit{strictly positive} if $f(x)>0$ for $m$-almost every $x\in X$.
 The semigroup $(e^{-tL})_{t\ge0}$ is called \textit{positivity preserving} if $e^{-t L}$ maps non-negative functions to non-negative functions for each $t> 0$; it is called \textit{positivity improving} if
 $$ e^{-t L} f\; \mbox{ is strictly positive}$$
for any positive $f$ and all $t>0$. Note that $\sigma$-finiteness of the measure $m$ is a necessary condition for the semigroup to be positivity improving.

\medskip

\textbf{Remark}. As is well known, the semigroup is positivity preserving if and only if the associated symmetric form $Q$ satisfies the first Beurling-Deny criterion (see e.g.\ \cite{BH,Fuk}). A positivity preserving semigroup is positivity improving if it has further irreducibility properties; see Appendix~\ref{Irreducibility} for further details.

\bigskip

The following result on positivity preserving semigroups is the crucial new ingredient of our considerations and the main result of this section.

\begin{theorem} \label{positivitypreserving}  Let $L$ be selfadjoint and bounded below in $L^2 (X,m)$ with infimum of the spectrum $E_0$. Assume that the associated semigroup $(e^{-t L})_{t\ge0}$ is positivity preserving. Then
 $$ \lim_{t\to \infty} \frac{\log \langle f, e^{-t L} g\rangle}{t} = - E_0$$
for all strictly positive $f$ and $g$ in $L^2 (X,m)$.
\end{theorem}
\begin{proof}
Without loss of generality we can assume that $E_0 = 0$. We will show two inequalities:

\smallskip

As $e^{- t L}$ is positivity preserving, we have
 $$ 0 \le \langle f,e^{-t L} g\rangle \le \|e^{-t L}\| \|f\| \|g\|.$$
As $E_0 =0$, we have $\|e^{-t L}\|\le 1$ for all $t\ge 0$ and hence
 $$\limsup_{t\to \infty} \frac{\log \langle f, e^{-t L} g\rangle}{t} \le 0$$
follows. To show the reverse inequality we need the assumption that $f$ and $g$ are strictly positive.  Thus,
 $$ h := \min\{f,g\}\;\: \mbox{satisfies} \;\: h(x) >0 \;\:\mbox{for $m$-almost every $x\in X$.}$$
Therefore,
 $$\calA := \set{u\in L^2 (X,m)}{0 \le u \le h}$$
is total in $L^2 (X,m)$ (i.e., the linear span of $\calA$ is dense). Thus,
 $$ 0= E_0 = \inf_{u\in \calA} E_u, \quad (\clubsuit) $$
where $E_u := \inf \support (\rho_u)$.
As $e^{-t L}$ is positivity preserving and obviously
 $$f,g\ge \min\{f,g\} = h \ge u \ge 0$$
for any $u\in \calA$, we obtain
 $$\langle f, e^{-t L} g\rangle \ge \langle h,e^{-t L} h\rangle \ge \langle u,e^{-t L} u\rangle$$
for any $u\in \calA$. Combined with (b) of Theorem~\ref{spectralversion} this gives
 $$\liminf_{t\to \infty} \frac{ \log \langle f, e^{-t L} g\rangle}{t} \ge \lim_{t\to \infty} \frac{\log \langle u, e^{-t L} u \rangle}{t} = - E_u$$
for any $u\in \calA$. By $(\clubsuit)$ we obtain the desired inequality
 $$ \liminf_{t\to \infty} \frac{ \log \langle f, e^{-t L} g\rangle}{t} \ge 0,$$
and the theorem is proven.
\end{proof}

In the case of positivity improving semigroups, the previous result can be strengthened in that the assumption of strict positivity on $f$ and $g$ can be weakened.  This is spelled out in the next result.  This result can be seen as an integrated version of Li's theorem. It does not require existence of kernels.

\begin{theorem}\label{positivityimprovingversion}
Let $L$ be selfadjoint and bounded below in $L^2 (X,m)$ with infimum of the spectrum $E_0$. Assume that the associated semigroup $(e^{-t L})_{t\ge0}$ is positivity improving. Then
 $$ \lim_{t\to \infty} \frac{\log \langle f, e^{-t L} g\rangle}{t} = - E_0$$
for all positive $f$ and $g$ in $L^2 (X,m)$.
\end{theorem}
\begin{proof}
As $e^{-t L}$ is positivity improving, the functions
$e^{- L} f$, $e^{- L} g$ are strictly positive. Clearly,
 $$ \langle f, e^{- t L} g\rangle = \langle e^{-L} f, e^{- (t - 2)L} e^{-L} g\rangle $$
holds for any $t>2$. Now, the theorem follows easily from Theorem~\ref{positivitypreserving}.
\end{proof}

We note the following consequence of the theorem and part (b) of Theorem~\ref{spectralversion}.

\begin{coro}
Consider $L$ as in the previous theorem. Let $f \in L^2 (X,m)$ be positive. Then
 $$ E_0 = \inf \support (\rho_f), $$ \smallbds
and, in particular,
 $$ \ind_{ [E_0,E_0+ \delta)} (L) f\ne 0$$ \smallbds
for any $\delta >0$.
\end{coro}

\textbf{Remarks.} (a) The corollary is well known if $E_0$ is an eigenvalue. In this case there exists a (unique) almost everywhere positive normalized eigenfunction to $E_0$. This eigenfunction has then a non-vanishing inner product with any positive $f$ and this easily implies the corollary.

(b) For strictly positive $f$ the corollary will also hold if
 the semigroup is only assumed to be positivity preserving (as can be seen from Theorem~\ref{positivitypreserving}).

(c) If the semigroup is only assumed to be positivity preserving, then the corollary (and the preceeding theorem) will in general be false: consider the direct sum of two positivity improving semigroups on disjoint sets and assume that the infima of the spectra of the two generators are different.

\medskip

Large parts of the basic theory of positivity preserving semigroups do not depend on the selfadjointness of the generator. For this reason it is remarkable that our results crucially depend on this selfadjointness: In the following example it is shown that Theorem~\ref{positivityimprovingversion} is not true in general without the selfadjointness of $L$.

\begin{example}
Let $\N = \{1,2,\ldots\}$ be the set of natural numbers and $\ell_2 := \ell^2 (\N)$ the associated $\ell^2$ space (with $m(\{x\}) = 1$ for all $x\in \N$).
Let $L$ be the left shift on $\ell_2$, i.e, $Lx = (x_2,x_3,\dots)$ for all $x = (x_1,x_2,x_3,\dots) \in \ell_2$. Observe that, as a positivity preserving operator, $L$ generates a positivity preserving semigroup $(e^{tL})_{t\ge0}$ on $\ell_2$ which, however, is not positivity improving. (Note that here we consider $L$ and not $-L$ as a generator.)

Let $\mu\in(0,1)$. Then $y_\mu := (\mu^k)_{k\in\N} \in \ell_2$. We define a bounded positivity preserving operator $L_1$ on $\ell_2$ by
 $$L_1x := Lx + x_1 y_\mu;$$
then one easily sees that $L_1$ generates a positivity improving semigroup $(e^{tL_1})_{t\ge0}$ on $\ell_2$. For $\lambda\in(0,1)$, $\lambda \ne 2\mu$, a straightforward computation shows that the function $f \from [0,\infty) \to \ell_2$,
 $$f(t) := e^{\lambda t} y_\lambda
           + \frac{\lambda}{\lambda-2\mu}(e^{\lambda t}-e^{2\mu t}) y_\mu $$
satisfies the differential equation $f'(t) = L_1f(t)$ for $t\ge0$, and $f(0) = y_\lambda$. Therefore,
 $$e^{tL_1} y_\lambda = f(t)$$ \smallbds
for all $t\ge0$.

Now suppose that $\mu < \frac12$. Then for $\lambda \in (2\mu,1)$ and any positive $x \in \ell_2$ we obtain
 $$ \lim_{t\to \infty} \frac{\log \langle x, e^{tL_1} y_\lambda\rangle}{t}
 = \lambda. $$
In particular, the limit depends on the choice of $\lambda$, so no analogue of Theorem~\ref{positivityimprovingversion} can be true for the operator $L_1$.
\end{example}

\section{Semigroups with kernels}\label{Semigroups}

In this section we further specialize the setting of the last section by assuming existence of (pointwise consistent) kernels. More precisely, we assume, for $f\in L^2 (X,m)$ and $t>0$, that
 $$ e^{-t L} f (x) = \int_X p_t (x,y) f(y) \, dm(y)\,\ \mbox{for $m$-almost every}\ x\in X, $$
for a measurable function
 $$p \from (0,\infty)\times X\times X \longrightarrow (0,\infty)$$ \smallbds
satisfying
\begin{itemize}
\item[(K1)] $p_t (x,y) = p_t (y,x)$ for all $x,y\in X$,
\item[(K2)] $p_t (x,\cdot) \in L^2 (X,m)$ for any $x\in X$ and $t>0$,
\item[(K3)] $p_{t+ s} (x,y) = \int_X p_t (x,z) p_s (z,y) \, dm (z)$ for all $x,y\in X$ and $t,s>0$.
\end{itemize}
Note that $p$ is positive everywhere and accordingly $e^{-t L}$ is positivity improving.


\medskip

\textbf{Remarks.} (a) In concrete situations, strict positivity of $p$ may be achieved from non-negativity of $p$ by removing a set of measure zero from $X$.  Then our results below remain valid (after removing this set).
See also Remark~(c) at the end of this section.

(b) At the end of the section we will show that the existence of a kernel with these properties is automatically true in certain (rather general) topological situations.

\medskip

Under the above assumptions, we can combine the results of the previous section with ideas developed in the context of manifolds in \cite{CK,Sim} to obtain the following result. In fact, part~(a) of the result is a rather direct adaption of the proof in \cite{Sim}. Part~(b) is then a consequence of (a) combined with Theorem~\ref{positivityimprovingversion}. Indeed, Theorem~\ref{positivityimprovingversion} is crucial for dealing with the situation when the infimum of the spectrum is not an eigenvalue.

\begin{theorem}\label{kernelversion}
Let $L$ be a selfadjoint operator in $L^2 (X,m)$ with infimum of the spectrum given by $E_0 >-\infty$. Assume that the semigroup $(e^{-tL})_{t\ge0}$ has a kernel\/ $p$ as above. Then the following holds:

\textup{(a)} There exists a unique measurable function $\Phi \from X \to [0,\infty)$ such that
 $$ e^{t E_0} p_t (x,y) \to \Phi (x) \Phi (y) \qquad (t\to \infty)$$
for all $x,y\in X$. If $E_0$ is an eigenvalue, then the function $\Phi$ is a strictly positive normalized eigenfunction to $E_0$. If $E_0$ is not an eigenvalue, then the function $\Phi$ vanishes everywhere.

\textup{(b)} For all $x,y\in X$ the convergence
 $$ \frac{ \log p_t (x,y)}{t} \to - E_0 \qquad (t\to \infty)$$
holds.
\end{theorem}

\begin{proof}
Note that the uniqueness in (a) follows immediately from the convergence
statement and non-negativity of $\Phi$ by considering $x = y$.

Without loss of generality, assume that $E_0 = 0$. For any $x\in X$ we define the function $g_x\in L^2 (X,m)$ via the kernel at $t=1$ by
 $$ g_x (y) = p_1 (x,y).$$
Note that $g_x$ is positive everywhere by the assumption on $p$.
Moreover, the assumption (K3) on $p$ immediately gives
 $$ p_{t + 2} (x,y) = \langle g_x, e^{-t L} g_y\rangle. \quad (\spadesuit)$$
We now distinguish two cases:

\medskip

\textit{Case 1: $E_0 = 0$ is not an eigenvalue.} As $E_0$ is not an eigenvalue, the projection $P$ is zero. We set $\Phi \equiv 0$. Then by $(\spadesuit)$ and (a) of Theorem~\ref{spectralversion} we see that
 $$ p_t (x,y) = \langle g_x, e^{- (t-2) L} g_y\rangle \to \langle g_x, P g_y\rangle = 0 = \Phi (x) \Phi (y)$$
as $t\to\infty$, for all $x,y\in X$. This gives (a) of the theorem in this case. Moreover, $(\spadesuit)$ and Theorem~\ref{positivityimprovingversion} give
 $$ \frac{ \log p_t (x,y)}{t} = \frac{ \log \langle g_x, e^{-(t-2)L} g_y\rangle}{t} \to - E_0.$$
This shows the desired statements in this case.

\bigskip

\textit{Case $2$: $E_0=0$ is an eigenvalue.} By general principles (see e.g.\ Section XIII.12 of \cite{RS}), there then exists a unique normalized eigenfunction $\Psi$ that is positive $m$-almost everywhere and satisfies
 $$P = \langle \Psi, \cdot\rangle \Psi.$$
Observe that
 $$\Phi (x) := \langle g_x, \Psi\rangle = \int p_1 (x, y) \Psi (y) \, dm(y) = e^{-L} \Psi(x) = \Psi(x)$$
for $m$-almost every $x\in X$, so that $\Phi$ is a representative of $\Psi$.
Moreover, by the assumption on $p$ the function $g_x$ is strictly positive for all $x\in X$ and hence $\Phi(x)>0$ for all $x\in X$.
Finally, $(\spadesuit)$ and
(a) of Theorem~\ref{spectralversion} give
 $$ p_t (x,y) = \langle g_x, e^{- (t-2) L} g_y\rangle \to \langle g_x, P g_y\rangle = \langle g_x, \Psi \rangle \langle g_y, \Psi\rangle= \Phi (x) \Phi (y) $$
for all $x,y\in X$. This proves part (a) of the theorem in this case.
 Given strict positivity of $\Phi$, the convergence of the kernels gives easily part (b) of the theorem in this case.
\end{proof}

\textbf{Remark.} Part~(b) of the previous theorem can be seen as a rather general version of Li's theorem \cite{Li} (cf.\ the introduction).

\smallskip

Note that the theorem gives a characterization of whether $E_0$ is an eigenvalue:

\begin{coro}
Let $L$ be as in the theorem. Let $\Phi_t (x) := \left( e^{t E_0} p_t (x,x)\right)^{1/2}$. Then the following are equivalent:

\begin{itemize}
\item[(i)] $E_0$ is an eigenvalue.

\item[(ii)] The pointwise limit (for $t\to \infty$) of $\Phi_t$ is not the zero function in $L^2 (X,m)$.

\item[(iii)] There exists an $x\in X$ such that $\lim_{t\to\infty} \Phi_t (x) \ne 0$.
\end{itemize}
\end{coro}

We conclude the section with the already mentioned discussion of existence of kernels in a topological setting.

\smallskip

\begin{prop} \label{existence-kernel} Let $X$ be a locally compact separable metric space and $m$ a Radon measure on $X$ with full support. Let $L$ be selfadjoint and bounded below in $L^2 (X,m)$ with positivity improving semigroup $(e^{-t L})_{t\ge 0}$.  Assume that $e^{-t L} $ maps $L^2 (X,m)$ into $C(X)$, the set of continuous functions on $X$. Then there exists a (unique) measurable function $p \from (0,\infty)\times X\times X \to [0,\infty)$ that is continuous separately in each $X$ variable, fulfills (K1), (K2), (K3) and satisfies
 $$ e^{-t L} f (x) = \int_X p_t (x,y) f(y) \,dm (y)$$
for all $x\in X$ and $f\in L^2 (X,m)$ (where $e^{-t L} f$ denotes the unique continuous representative). Moreover, there exists a closed subset $M$ of zero $m$-measure such that $p_t (x,y)>0$ for any $t>0$ whenever $x\notin M$ and $y\notin M$ hold.
\end{prop}
\begin{proof}
Uniqueness of $p$ is clear from the continuity properties.  We prove existence:
For any continuous $\varphi \from X \to \R$ with compact support and any $t>0$, the operator $M_\varphi \+ e^{-t L}$ maps $L^2 (X,m)$ into $C_\b (X)$, the bounded continuous functions on $X$ (where $M_\varphi $ denotes the operator of multiplication by $\varphi$).
The closed graph theorem shows that $M_\varphi \+ e^{-t L} \from L^2 (X,m) \to C_\b (X)$ is continuous.  Then the Riesz representation theorem provides for each $x\in X$ and $t>0$ existence of a unique $k_t (x,\cdot)\in L^2 (X,m)$ with
 $$ e^{-t L} f(x) = \scpr{k_t(x,\cdot)}{f} = \int_X k_t (x,y) f(y) \,dm(y) $$
for all $f \in L^2 (X,m)$.

We now define
$$ p_t (x,y) := \int_X k_{t/2} (x,z) k_{t/2} (y,z) \,dm (z)
              = \scpr{k_{t/2} (x,\cdot)}{k_{t/2} (y,\cdot)} $$
for all $t>0$ and $x,y\in X$.
By construction, each $p_t$ is symmetric, hence (K1) is satisfied.  By construction, $p_t$ is continuous in $x$ and $y$ separately.

\smallskip

To show measurability of $p$, we first note that $(t,x)\mapsto k_t (x,\cdot)$ is weakly continuous.  Indeed, for any $f\in L^2 (X,m)$ and arbitrary fixed $s>0$ we have the continuity of
$$ (s,\infty)\times X \ni (t,x)
   \mapsto e^{-t L} f(x) = e^{-s L} \bigl( e^{- (t- s)L} f\bigr) (x) $$
as $t \mapsto e^{- (t - s)L} f \in L^2 (X,m)$ is continuous by the strong continuity of the semigroup and $(x,g)\mapsto e^{- s L} g (x)$ is continuous by the continuity of $M_\varphi \+ e^{-t L} \from L^2 (X,m)\to C_b (X)$ for all continuous $\varphi$ with compact support. Now, by the separability assumption, the space $L^2 (X,m)$ has an (at most) countable orthonormal basis $(e_n)$. Hence,
$$ p_t (x,y) = \sum_n \scpr{ k_{t/2} (x, \cdot)}{e_n} \scpr{e_n}{ k_{t/2} (y, \cdot) } $$
is measurable as a countable sum of continuous functions.  This concludes the proof of the measurability statement.

We now show that $p$ is indeed a kernel of the semigroup and satisfies (K2).  Let $s,t>0$ and $x\in X$. Then for any $f\in L^2 (X,m)$ we have
\begin{align*}
\scpr{k_{t+s}(x,\cdot)}{f}
 &= e^{-(t+s) L} f(x) = e^{-s L} e^{-t L} f(x) \\
 &= \scpr{k_s(x,\cdot)}{e^{-t L} f} = \scpr{e^{-t L} k_s(x,\cdot)}{f},
\end{align*}
so we obtain
 $$k_{t+s}(x,y) = e^{-t L} k_s(x,\cdot)(y) = \scpr{k_s(x,\cdot)}{k_t(y,\cdot)}$$
for a.e.\ $y\in X$, by the definition of $k_t(y,\cdot)$. In particular, $k_t (x,\cdot) = p_t (x,\cdot)$ a.e.\ by the definition of $p_t$. Then (K2) follows as well as
 $$ \int_X k_t (x,y) f(y) \,dm(y) = \int_X p_t (x,y) f(y) \,dm(y) = e^{-t L} f (x) $$
for any $f\in L^2 (X,m)$ and any $x\in X$.
Moreover, we now see that $p$ satisfies
 $$ p_{t+s}(x,y) = \scpr{p_s(x,\cdot)}{p_t(y,\cdot)} $$
for all $x,y\in X$ as, for fixed $x$, both sides are continuous in $y$ and agree for a.e.\ $y\in X$ as observed above. (Recall that $y \mapsto p_t(y,\cdot)$ is weakly continuous.)
By the symmetry of $p_s$ this is (K3), so $p$ has the desired properties.

\smallskip

It remains to show the statement on $M$:  We first show that for $x\in X$ the following alternative holds: Either $p_t (x,\cdot)$ is strictly positive for all $t>0$ or it vanishes identically for all $t>0$.

Indeed, if $p_t (x,\cdot)$ is not strictly positive for a $t>0$ and an $x\in X$, then there exists a non-negative $f$ in $L^2 (X,m)$ with $f \neq 0$ and $e^{-t L} f (x) = 0$, and hence
 $$ 0 = e^{-t L} f(x) = e^{- s L} (e^{- (t -s) L} f) (x) = \scpr{ p_s (x,\cdot)}{ e^{- (t -s)L} f}$$
 for all $0 < s < t$.
 As the semigroup is positivity improving, $e^{- (t - s)L} f$ is strictly positive for any $0 < s < t$ and we infer that $p_s (x,\cdot) = 0$ in $L^2 (X,m)$ for $0 < s < t$. By continuity of $p_s (x,\cdot)$ and (K3) we obtain $p_s (x,y) =0$ for all $y\in X$ and $s >0$.

We now define $M$ to be the set of $x\in X$ for which $p_1 (x,\cdot)$ vanishes identically. Then $M$ is closed by the continuity properties of $p_1$. Moreover, it has measure zero as $e^{- L}$ is positivity improving. Finally, by the alternative just discussed and (K1) and (K3) we have for $t>0$ and $x\notin M$ and $y\notin M$
that
$$p_t (x,y) = \scpr{p_{t/2} (x,\cdot)}{p_{t/2}(y,\cdot)} >0$$
holds. This gives the desired statement for $M$.
\end{proof}

\textbf{Remarks.} (a) The assumption that $e^{-t L}$ maps into $C(X)$ is satisfied whenever the domain of $L^N$ is contained in $C(X)$ for some $N>0$ (which holds true in `smooth' situations by local Sobolev theorems).

(b) The set $M$ appearing in the statement of the proposition can in general not be avoided (as can easily be seen by considering singular perturbations).

(c) In the situation of the proposition one can remove $M$ from $X$ without changing the associated $L^2$ space (as $M$ has measure zero). With $X$ replaced by $X\setminus M$ one is then exactly in the situation of Theorem~\ref{kernelversion} (compare with the first remark in this section).

(d) Related results on existence of kernels can be found e.g.\ in \cite{JS}. More specifically, Theorem~1.12 of that paper gives existence of kernels for sub-Markovian semigroups on $L^p$ whose adjoint is again sub-Markovian and whose range satisfies some semicontinuity properties. So, under the additional assumption of contractivity on $L^\infty$, the previous proposition could also be inferred from \cite{JS}.



\section{Admissible potentials}\label{Application}

In this section we assume that we are given a positivity improving selfadjoint semigroup. In this situation one can then study perturbations by potentials $V \from X\to \R$. Here, two types of perturbations are of particular interest. These are perturbations that are positive (or, more generally, bounded below) and perturbations that are negative (or, more generally, bounded above). It turns out that perturbations which are bounded below are `essentially harmless'. Details providing a more precise version of this statement are discussed in the appendix.

\smallskip

Here, we consider the much more subtle situation of perturbations arising from negative potentials. Our theorem below essentially generalises a result from Cabr\'{e} and Martel \cite{CM} for the heat equation on smooth bounded subdomains of Euclidean space to all selfadjoint positivity improving semigroups. Note that below the addition of a negative potential is performed by subtracting a positive potential.

\bigskip

Let $L$ be a selfadjoint operator in $L^2 (X,m)$ that is bounded below and assume that the semigroup $(e^{-tL})_{t\ge0}$ is positivity improving.
Let $V \from X\to [0,\infty)$ be a measurable potential. We then define the ``generalised ground state energy'' of $L-V$ by
 $$ \lambda_0(L,V) := \inf \set{Q(u,u)-\|V^{1/2}u\|_2^2}{u\in D(Q),\ \|u\|_2=1}. $$
Note that $\lambda_0(L,V) = -\infty$ is possible, for example, if $V^{1/2}u \notin L^2 (X,m)$ for some $u\in D(Q)$ or, more generally, whenever $V$ is not form bounded with respect to $Q$ with bound less or equal to one.

The crucial observation is that, for every $t\ge0$, the sequence $(e^{- t (L - \Vk)})$ of positive operators is increasing, i.e., $(e^{- t (L - \Vk)}f)$ is an increasing sequence in $L^2 (X,m)$ for each positive $f \in L^2 (X,m)$. (This is an easy consequence of the Trotter product formula.)
The potential $V$ is called \emph{admissible} if the strong limits
 $$S_V (t) := \operatorname{s-}\hspace{-0.45em}\lim\limits_{k\to\infty}
              e^{- t (L - \Vk)}$$
exist and form a $C_0$-semigroup, i.e., satisfy
\begin{itemize}
\item $S_ V(t + s) = S_V (s) S_V (t)$ for all $s,t>0$ and
\item $S_V (t)\to I$ strongly for $t\to 0$.
\end{itemize}
It is well known that admissibility of $V$ follows if the $S_V (t)$ are exponentially bounded (see the proof of the subsequent proposition as well). We refer to \cite{voi86,voi88} for the notion of admissibility.

\begin{prop}\label{admissible}
Let $L$ be a selfadjoint operator in $L^2 (X,m)$ that is bounded below and assume that the semigroup $(e^{-tL})_{t\ge0}$ is positivity improving. Let $V \from X \to [0,\infty)$ be measurable and let $E\in\R$. Then the following assertions are equivalent:
\begin{itemize}
\item[(i)] $V$ is admissible and $\|S_V(t)\| \le e^{-Et}$ for all $t\ge0$.

\item[(ii)] There exist $M>0$ and positive functions $f,g \in L^2(X,m)$ such that
 $$ \langle f, e^{-t(L-\Vk)}g \rangle \le Me^{-E t} $$
for all $t\ge0$ and $k\in\N$.

\item[(iii)] The inequality $\lambda_0(L,V) \ge E$ holds; in other words, $V+E \le L$ holds \emph{in the form sense}, i.e.,
 $$ \|V^{1/2}u\|_2^2 + E\|u\|_2^2 \le Q(u,u) $$
for all $u \in D(Q)$.
\end{itemize}
\end{prop}
\begin{proof}
Obviously, $\Vk$ is bounded and non-negative for any $k\in \N$. The assumptions on $L$ then give that the operator $L - V \wedge k$ is bounded below and generates a positivity improving semigroup. Thus,
if $f,g \in L^2(X,m)$ are positive and $k\in\N$, then
 $$ E_0(L-\Vk) = -\lim_{t\to \infty} \frac{\log \langle f, e^{-t(L-\Vk)} g\rangle}{t} $$
by Theorem~\ref{positivityimprovingversion}. Therefore, property~(ii) implies $E_0(L-\Vk) \ge E$ for all $k\in\N$, and thus
\begin{itemize}
\item[(ii')] $L-\Vk \ge E$ in the form sense for all $k\in\N$.
\end{itemize}
Conversely, if (ii') holds, then (ii) is valid with any positive functions $f,g \in L^2(X,m)$ and $M = \|f\|_2\|g\|_2$.

Now, the equivalence of (i), (ii') and (iii) is shown in Proposition~5.7 of \cite{voi86}. There the proof is given for the heat semigroup on $\R^n$ only, but literally the same proof carries over to the general case. This proves the desired equivalence.

For illustrative purposes we actually give a proof of ``(iii)$\Rightarrow$(i)'' here. While slightly longer than the proof in \cite{voi86}, our proof seems to be more elementary:  If (iii) holds, it will also hold with $V$ replaced by $\Vk$ for any $k\in \N$. As $\Vk$ is bounded, this gives
$$\|e^{-t(L-\Vk)}\| \le e^{-Et}$$
 for all $t\ge0$ and $k\in\N$.
Since for every positive $f \in L^2 (X,m)$ the sequence $(e^{- t (L - \Vk)}f)$ is increasing, it follows that the limit in the definition of $S_V(t)$ exists for each $t\ge0$ and $\|S_V(t)\| \le e^{-Et}$ for all $t\ge0$. Then, it is not hard to see that $S_V(t + s) = S_V (t) S_V (s)$ for all $s,t\geq 0$.
It remains to show that $S_V(t)f \to f$ in $L^2(X,m)$ as $t\to0$, for all $f \in L^2(X,m)$. By linearity we can assume without loss of generality that $f\ge0$. Note that $0 \le u(t) := e^{-tL}f \le S_V(t)f =: u_V(t)$ for all $t\ge0$. It follows that
 $$ \| u_V(t) - u(t) \|_2^2
    = \|u_V(t)\|_2^2 + \|u(t)\|_2^2 - 2\langle u_V(t), u(t) \rangle
  \le \|u_V(t)\|_2^2 - \|u(t)\|_2^2 $$
for all $t\ge0$. Moreover, $\|u(t)\|_2 \to \|f\|_2$ as $t\to0$ (since $u(t) \to f$ in $L^2(X,m)$) and $\|u_V(t)\|_2 \le e^{-Et}\|f\|_2$ for all $t\ge0$. We conclude that $u_V(t) - u(t) \to 0$ and hence $u_V(t) \to f$ in $L^2(X,m)$ as $t\to0$.
\end{proof}

It is possible to reformulate (parts of) the preceding theorem in terms of solutions of a corresponding abstract Cauchy problem and to thereby characterize when $\lambda_0 (L,V)>-\infty$ holds.  This is done next.
\smallskip

Let $\rho \in L^2(X,m)$ be strictly positive and let $f \in L^1(X,\rho m)$.
We say that $u \from [0,\infty) \to L^1(X,\rho m)$ is an \emph{approximated solution with respect to $\rho$} of the initial value problem
 $$ u'(t) + Lu(t) = Vu(t) \quad (t>0), \qquad u(0) = f \qquad (\heartsuit) $$
if there exists a sequence $(f_k)$ in $L^2(X,m)$ such that $0\le f_k\uparrow f$ $m$-almost everywhere, $e^{-t(L-\Vk)}f_k \to u(t)$ in $L^1(X,\rho m)$ as $k\to\infty$, for all $t\ge0$, and $u(t) \to f$ as $t\to0$. Usually, we will suppress the dependence on $\rho$ when speaking about approximated solutions. By definition, any approximated solution is continuous at $t=0$.
We also note that $L^2(X,m) \subseteq L^1(X,\rho m)$ since $\rho \in L^2(X,m)$.

\begin{theorem}\label{adm-cor}
Let $L$ be a selfadjoint operator in $L^2 (X,m)$ that is bounded below and assume that the semigroup $(e^{-tL})_{t\ge0}$ is positivity improving. Let $V \from X \to [0,\infty)$ be measurable. Then $\lambda_0(L,V) > -\infty$ if and only if $(\heartsuit)$ has an approximated solution that is exponentially bounded in $L^1(X,\rho m)$ for some strictly positive $\rho \in L^2(X,m)$ and some positive $f \in L^1(X,\rho m)$. Moreover, if $\lambda_0 (L,V)>-\infty$, then, for any positive $f \in L^2 (X,m)$, there exists an approximated solution $u\from [0,\infty)\longrightarrow L^2 (X,m)$ which is continuous and exponentially bounded.
\end{theorem}
\begin{proof}
For the proof of sufficiency, let $u$ be the presumed approximated solution. Let $(f_k)$ be the sequence in $L^2(E,m)$ approximating $f$, with $f_1 \ne 0$ without loss of generality, and let $M>0$ and $E\in\R$ be such that $\|u(t)\|_{L^1(X,\rho m)} \le Me^{-Et}$ for all $t\ge0$. Then, from the monotonicity of $(e^{- t (L - \Vk)})$, it follows that
 $$ \langle \rho, e^{-t(L-\Vk)}f_1 \rangle
    \le \langle \rho, u(t) \rangle
    = \|u(t)\|_{L^1(X,\rho m)}
    \le Me^{-Et} $$
for all $t\ge 0$ and $k\in\N$.
Since $f_1$ is positive, it follows from Proposition~\ref{admissible} that $\lambda_0(L,V) \ge E > -\infty$.

Necessity is clear: If $\lambda_0(L,V) > -\infty$, then $V$ is admissible by Proposition~\ref{admissible}. Hence, for every positive $f \in L^2(X,m)$, the map $t\mapsto S_V (t) f$ provides an approximated solution to $(\heartsuit)$ that is continuous and exponentially bounded in $L^2(X,m)$.
\end{proof}

The above characterization of $\lambda_0(L,V) > -\infty$ is phrased in terms of the existence of approximated solutions. It is possible to replace approximated solutions by suitable weak solutions. This is discussed next.

\smallskip

We say that $u\from[0,\infty)\times X\to\R$ is a \emph{weak solution with respect to $\rho$} of $(\heartsuit)$ if $(1+V)u\rho \in L^1((0,T)\times X)$ for any $T>0$ and
$$ \int_0^\infty \! \int u(t)\bigl(L\varphi(t)-\varphi'(t)\bigr)\,dm\,dt
   = \int f\varphi(0)\,dm + \int_0^\infty \! \int Vu(t)\varphi(t)\,dm\,dt $$
for all compactly supported $C^\infty$ functions $\varphi\from[0,\infty)\to D(L)$ satisfying
$$ |\varphi(t)|, |\varphi'(t)|, |L\varphi(t)| \le c\rho \qquad (t\ge0) $$
for some $c\ge0$.
Such functions $\varphi$ will be called \emph{test functions with respect to $\rho$}.
Usually we will suppress the dependence on $\rho$ when speaking about weak solutions.

\medskip

\textbf{Remarks.}
(a) It suffices to assume the above for all compactly supported $C^\infty$ functions $\varphi\from[0,\infty)\to\bigcap_{k\in\N} D(L^k)$. To see this, replace $\varphi(t)$ with $e^{-\eps L} \varphi(t)$ and let $\eps\to0$.

(b) The definition includes the setting of Cabr\'e and Martel \cite{CM} as a special case. In fact, they have the assumptions that $u, Vu\rho \in L^1((0,T)\times X)$, which are more restrictive since they work with a bounded weight $\rho$.

\medskip

\begin{lemma}
\textup{(a)} Assume that there exists $c\ge1$ such that $e^{-t L}\rho \le c\rho$ for all $t\in[0,1]$.
If $0\le f\in L^1(X,\rho m)$ and $(\heartsuit)$ has a non-negative weak solution $u\from [0,\infty) \longrightarrow L^1 (X,\rho m)$ which is continuous at $t=0$, then there exists an approximated solution $v$ with $0\le v\le u$.

\textup{(b)} Assume that there exists $c>0$ such that $e^{-t L}\rho \ge c\rho$ for all $t\in[0,1]$ and $e^{-L}\rho \le \frac1c \rho$.
Then every locally bounded approximated solution $u\from [0,\infty) \longrightarrow L^1 (X, \rho m)$ is a weak solution.
\end{lemma}
\begin{Proof}
(a) Let $(f_k)$ be a sequence in $L^2(X,m)$ such that $0\le f_k\uparrow f$ $m$-almost everywhere. For $k\in\N$ we define $u_k(t) := e^{-t (L-\Vk)}f_k$. We are going to show that $u_k \le u$ for all $k\in\N$; then it follows that the limit $v := \lim_{k\to\infty} u_k$ exists and $v\le u$. Moreover, the convergence $v(t) \to f$ as $t\to0$ can then be seen as follows: given a fixed $k\in\N$ we have $u_k(t) \le v(t) \le u(t)$ for all $t>0$ and hence $|v(t)-f| \le \max\bigl\{ u(t)-f, f-u_k(t) \bigr\}$, so
 $$ \| v(t)-f \|_{L^1(\rho m)} \le \| u(t)-f \|_{L^1(\rho m)} + \| f-f_k \|_{L^1(\rho m)} + \| f_k-u_k(t) \|_{L^1(\rho m)}. $$
The middle term on the right hand side is small for large $k$, and the first and third terms converge to $0$ as $t\to0$. Thus, $v$ is an approximated solution.

For the proof of $u_k \le u$ fix $k\in\N$, a measurable function $0\le \varphi_0\le\rho$ and a compactly supported function $0\le\psi\in C^\infty(0,\infty)$. We define $\varphi\from[0,\infty)\to D(L)$ by
 $$ \varphi(t) := \int_0^t \psi(t-r) e^{-r (L-\Vk)}\varphi_0\,dr. $$
Then $\varphi'(t) = \int_0^t \psi'(t-r) e^{-r (L-\Vk)}\varphi_0\,dr$ as $\psi(0) = 0$.
Note that the assumtion implies $e^{-rL}\rho \le c^{1+r}\rho$ for all $r\ge0$.
Therefore, $e^{-r (L-\Vk)}\varphi_0 \le e^{-r (L-k)}\varphi_0 \le e^{rk} c^{1+r}\rho$, and we infer that there exists $c_1>0$ such that
 $$ \varphi(t),|\varphi'(t)| \le c_1 \int_0^t c^{1+r} e^{rk} \,dr \cdot \rho $$
for all $t \ge 0$. On the other hand, $\varphi(t) = \int_0^t \psi(r) e^{-(t-r) (L-\Vk)}\varphi_0\,dr$ and hence $\varphi'(t) = \psi(t)\varphi_0 - (L-\Vk)\varphi(t)$, so
 $$ |L\varphi(t)| \le \|\psi\|_\infty\rho + k\varphi(t) + |\varphi'(t)|. $$

We fix $T>0$ and define $\varphi_T \from [0,\infty) \to D(L)$ by
 $$ \varphi_T(t) := \varphi(T-t)\ \ \text{for}\ 0\le t\le T, \qquad \varphi_T(t) := 0 \ \ \text{for}\ t>T. $$
By the above estimates, $\varphi_T$ is a test function with respect to $\rho$.
Moreover, $ \varphi_T'(t) = -\psi(T-t)\varphi_0 + (L-\Vk)\varphi(T-t) $
for all $0\le t\le T$, so from $u$ being a weak solution we obtain
\begin{gather*}
\int_0^T \! \int u(t) \bigl( (\Vk) \varphi(T-t) + \psi(T-t)\varphi_0 \bigr)\,dm\,dt \\
\qquad
= \int f \varphi(T)\,dm + \int_0^T \! \int Vu(t) \varphi(T-t)\,dm\,dt.
\end{gather*}
Therefore, since $u,\varphi\ge0$,
 $$ \int_0^T \! \int u(t) \psi(T-t) \varphi_0 \,dm\,dt \ge \int f \varphi(T)\,dm. $$
By the definition of $\varphi(T)$, the estimate $f \ge f_k$ and the self-adjointness of $e^{-t (L-\Vk)}$ it follows that
\begin{align*}
\int_0^T \psi(T-t) \int u(t) \varphi_0 \,dm\,dt
 &\ge \int_0^T \! \int f_k \psi(T-t) e^{-t (L-\Vk)}\varphi_0\,dm\,dt \\
 &= \int_0^T \psi(T-t) \int e^{-t (L-\Vk)} f_k\cdot\varphi_0\,dm\,dt.
\end{align*}
This inequality holds for any $0\le\psi\in C^\infty(0,\infty)$, any measurable function $0\le\varphi_0\le\rho$
and any $T>0$. Recalling $u_k(t) = e^{-t (L-\Vk)}f_k$, we conclude that $u\ge u_k$.

\smallskip

(b) Let $0\le f_k\uparrow f$ as in the definition of the approximated solution. For each $k\in\N$, the function $u_k$ defined by $u_k(t) := e^{-t (L-\Vk)}f_k$ is a strong $L^2(X,m)$-solution of the equation
 $$ u_k'(t) + Lu_k(t) = (\Vk)u_k(t) \quad (t>0), \qquad u_k(0) = f_k $$
and hence also a weak solution. Note that $u\rho \in L^1((0,T)\times X)$ for all $T>0$ since $u$ is locally bounded. Thus, for any test function $\varphi$ with respect to $\rho$, the Lebesgue convergence theorem yields
\begin{align*}
\int_0^\infty \! \int (\Vk)&u_k(t)\varphi(t)\,dm\,dt \\
= \int_0^\infty & \! \int u_k(t)\bigl(L\varphi(t)-\varphi'(t)\bigr)\,dm\,dt - \int f_k\varphi(0)\,dm \\
\longrightarrow \int_0^\infty & \! \int u(t)\bigl(L\varphi(t)-\varphi'(t)\bigr)\,dm\,dt - \int f\varphi(0)\,dm
\end{align*}
as $k\to\infty$.  It remains to show that $Vu\rho \in L^1((0,T)\times X)$ for all $T>0$; then it follows that $u$ is a weak solution.

Let $\varphi(t) :=  \psi(t) \cdot \int_0^1 e^{-s L}\rho\,ds \ge c\psi(t)\rho$ for all $t\ge0$, with $0\le\psi\in C_c^\infty[0,\infty)$. Then $L\varphi(t) = \psi(t)(\rho-e^{-L}\rho)$, so the assumptions imply that $\varphi$ is a test function with respect to $\rho$. By the monotone convergence theorem we conclude from the above that $Vu\rho \in L^1((0,T)\times X)$, which proves the assertion.
\end{Proof}

Combining Theorem~\ref{adm-cor} with the previous lemma, we immediately obtain the following.

\begin{coro}
Let $L$ be a selfadjoint operator in $L^2 (X,m)$ that is bounded below and assume that the semigroup $(e^{-tL})_{t\ge0}$ is positivity improving. Let $V \from X \to [0,\infty)$ be measurable. Then the following holds:

\textup{(a)} Assume that, for a positive $f \in L^1(X,\rho m)$ and for some strictly positive $\rho \in L^2 (X)$ satisfying $e^{-t L}\rho \le c\rho$ for a $c\ge1$ and all $t\in[0,1]$,
there exists a weak solution with respect to $\rho$ that is exponentially bounded in $L^1(X,\rho m)$ and continuous at $t=0$. Then $\lambda_0 (L,V) > -\infty$ holds.

\textup{(b)} Assume that there exists $c>0$ such that $e^{-t L}\rho \ge c\rho$ for all $t\in[0,1]$ and $e^{-L}\rho \le \frac1c \rho$. If $\lambda_0(L,V) > -\infty$ holds, then for every positive $f \in L^2 (X)$ there exists a positive weak solution (with respect to $\rho$) which is continuous and exponentially bounded in $L^2 (X,m)$.
\end{coro}

\medskip

The previous corollary can be understood as an abstract version of a result due to Cabr\'{e} and Martel \cite{CM}. This is discussed next.

\begin{example}
In \cite{CM}, Cabr\'{e} and Martel study the existence of positive exponentially bounded solutions of the heat equation with a potential in the following setting:
Let $X$ be a smooth bounded subdomain of $\R^n$ and let $\delta \in L^2(X)$ be defined by $\delta(x) = \dist(x,\partial X)$.
 Let $\Delta_D$ be the Dirichlet Laplacian in $L^2(X)$ and set $L := -\Delta_D$. Then $-L$ generates a positivity improving semigroup on $L^2(X)$.  Moreover, $L$ admits a ground state $f_0$, i.e., a normalized strictly positive eigenfunction to the ground state energy. By Section~3 of \cite{Dav}, the estimate
$c^{-1}\delta \le f_0 \le c\delta$
holds for some $c\ge1$, and this easily gives $e^{-t L} \delta \le c^2 \delta$ for all $t\ge 0$, as well as $e^{-t L} \delta \ge C \delta$ for all $t\in [0,1]$ with a suitable $C>0$.  This shows that $\delta$ satisfies all of the assumptions imposed on the density function $\rho$ above.

Let $V\from X \to [0,\infty)$ be measurable, let $0 \le f \in L^1(X,\delta m)$ be locally integrable and consider the initial value problem
 $$ \partial_t u - \Delta_D u = Vu, \quad u(0) = f. \qquad (\diamondsuit) $$
Then any weak solution of $(\diamondsuit)$ is continuous (as a function with values in $L^1(X,\delta m)$, see \cite{CM}) and the assumptions of the previous corollary are satisfied (with $\rho = \delta$). We thus obtain the following:
\begin{itemize}
\item[(1)] If $(\diamondsuit)$ has a positive weak solution that is exponentially bounded in $L^1(X,\delta m)$, for some positive $f \in L^1(X,\delta m)$, then $\lambda_0(-\Delta_D,V) > -\infty$.
\item[(2)] If $\lambda_0(-\Delta_D,V) > -\infty$, then $(\diamondsuit)$ has a positive weak solution that is exponentially bounded in $L^2(X)$, for every positive $f \in L^2(X)$.
\end{itemize}
With (1) and (2) we have obtained Theorem~1 of \cite{CM} as a special case of the previous corollary.
\end{example}

\bigskip

\textbf{Remarks.} (a) A proof of (1) and (2) in the previous example may also directly be based on Theorem~\ref{adm-cor}:
As is discussed in \cite{CM}, at the bottom of p.\,976, the initial value problem $ (\diamondsuit) $
has an approximated solution with respect to $\delta$ if and only if it has a positive weak solution with respect to $\delta$ (and if positive weak solutions exist, then the approximated solution is the minimal one). Now, given this equivalence, (1) and~(2) are immediate from Theorem~\ref{adm-cor}.

(b)
For Kolmogorov operators on weighted spaces $L^2(\R^n,\rho)$ a result analogous to the result in the previous example is known due to recent work of Goldstein/Goldstein/Rhandi (see Section~2 of \cite{GGR}).  Our Theorem~\ref{adm-cor} can also be used to reobtain the result of \cite{GGR} along lines similar to the ones given in part (a) of the remark.

\section{Examples}\label{Examples}

In this section we want to present some examples for which all of the assumptions of Section~\ref{Semigroups} are satisfied and hence all of the results of the previous sections hold.

\smallskip

We emphasize that non-negative potentials (satisfying some weak growth conditions) could be added to all of the operators $L$ below and the resulting operators would still generate positivity improving semigroups, see Corollary~\ref{Stability}.

\bigskip

\subsection{The Laplace-Beltrami operator on a manifold.}
The Laplace-Beltrami operator on a connected Riemannian manifold gives rise to a positivity improving semigroup with a continuous (and even $C^\infty$) kernel and the results of the previous sections hold.
For this example, validity of part (a) of Theorem~\ref{kernelversion} has already been proved by Chavel/Karp in \cite{CK}, with substantial later simplifications by Simon in \cite{Sim}. In fact, as mentioned already, part of our treatment in Section~\ref{Semigroups} is a rather direct adaption of these treatments.
Part (b) of Theorem~\ref{kernelversion} has been obtained by Li \cite{Li}. We refrain from discussing further details and refer to the mentioned literature.

\bigskip

\subsection{Laplacians on metric graphs}
Laplacians on metric graphs (also known as quantum graphs) have attracted considerable
interest in both physics and mathematics in recent years (see e.g.\ the articles \cite{Kuchment-04,Kuchment-05,KostrykinS-99b,KostrykinS-00b} and the conference proceedings \cite{BCFK, EKKSTG} and the references therein). While several variants and notations can be found in the literature, the basic setting is as follows (see \cite{LSS} for further details and proofs):
\begin{definition}
A \emph{metric graph} is a quintuple $\Gamma=(E,V,i,j,l)$ where
\begin{itemize}
\item $E$ (edges) and $V$ (vertices) are countable sets,
\item $l\from E\to (0,\infty)$ defines the length of the edges,
\item $i\from E\to V$ defines the initial point of an edge and
      $j\from \set{ e\in E }{ l(e)<\infty }\to V$ the end point for edges of finite
      length.
\end{itemize}
For $e\in E$ we set $X_e := \{e\}\times(0,l(e))$.
Moreover, we set $\overline{X_e} := X_e \cup \{i(e),j(e)\}$\, ($\overline{X_e} := X_e \cup \{i(e)\}$ if $l(e) = \infty$) and $X := X_\Gamma = V\cup\bigcup_{e\in E}X_e$.
\end{definition}
Thus, $X_e$ is essentially just the interval $(0,l(e))$, and the first
component is only added to force the mutual disjointness of the $X_e$'s. Then $\overline{X_e}$ can be identified with $[0,l(e)]$
 and this will be done tacitly in the following.
For simplicity, we will assume that all lengths $l(e)$ are uniformly bounded away from zero.

To introduce a metric structure on $X$, we say that $x\in X^N$ is a \emph{good polygon} if for every
$k\in\{1,\dots,N-1\}$ there is a unique edge $e\in E$ such that
$\{x_k,x_{k+1}\}\subseteq \overline{X_e}$. Using the usual distance in
$[0,l(e)]$ we get a distance $d$ on $\overline{X_e}$ and define
 $$ L(x) = \sum_{k=1}^Nd(x_k,x_{k+1}). $$

\medskip

Provided the graph is connected and the degree $d_v$ of every vertex $v\in V$ defined as
$$
d_v := \bigl|\set{ e\in E }{ v\in\{ i(e),j(e)\} }\bigr|
$$
is finite, a metric on $X$ is given by
$$
d(p,q) := \inf\set{ L(x) }{ x\mbox{ is a good polygon with } x_0=p\mbox{ and
}x_N=q } .
$$
If the graph is not connected (but $d_v$ is still finite for every $v\in V$), then, in this way, one can provide a metric on each connected component. Subsequently, we will assume finiteness of $d_v$ for all $v$ and equip the graph with the topology induced by these metrics on the components.

The Laplacian $L$ with Kirchhoff boundary conditions is now defined as the
operator corresponding to the form $Q$ with
$$
D(Q) := W^{1,2}(X), \:\; Q(u,v) := \sum_e(u_e'|v_e') ,
$$
where $u_e := u\circ \pi_e^{-1}$ is defined on $(0,l(e))$ with $\pi_e \from X_e \to (0,l(e))$ defined by $\pi_e ((e,s)) := s$, and
\begin{align*}
  W^{1,2}(X) &:= \biggl\{ u\in C(X) \biggm|
    \sum_{e\in E}\| u_e\|_{W^{1,2}}^2 =: \|u\|_{W^{1,2}}^2<\infty \biggr\}.
\end{align*}
Then $Q$ is a Dirichlet form and $-L$ generates a positivity preserving semigroup (see, e.g., \cite{Hae,LSS}; we refer to \cite{KKVW,LSchV} for more general boundary conditions).

\begin{prop} (Characterisation of positivity improvement)
Let $Q$ be as above and let $L$ be the associated operator. Then the semigroup $(e^{-tL})_{t\ge0}$ is positivity improving if and only if $X$ is connected.
\end{prop}
\begin{proof} It is clear that $e^{-t L}$ cannot be positivity improving if $\Gamma$ is not connected. The other implication (and much stronger results) follow immediately from the Harnack inequality presented in \cite{Hae} for connected graphs.
\end{proof}

Moreover, the existence of (continuous) kernels is known in this situation; see, e.g., \cite{Hae,LSS}. By the discussion at the end of Section~\ref{Semigroups} we obtain the following proposition.

\begin{prop}
Let $Q$ be as above and let $L$ be the associated operator. If $X$ is connected, then $e^{-t L}$ possesses a kernel\/ $p$ satisfying the conditions \textup{(K1)}, \textup{(K2)}, \textup{(K3)} of Section~\ref{Semigroups}.
\end{prop}

Given the previous propositions, as a consequence from Theorem~\ref{kernelversion} we obtain the following result for quantum graphs with Kirchhoff boundary conditions.

\begin{coro}
Let $\Gamma$ be a connected metric graph and let $L$ be the associated operator with Kirchhoff boundary conditions. Then the kernel\/ $p$ of $e^{-t L}$ satisfies
 $$ p_t (x,y) \to \Phi (x) \Phi (y) \:\:\mbox{and}\;\: \frac{\log p_t (x,y)}{t} \to -E_0 \qquad (t\to \infty)$$
for a unique $\Phi \from X \to [0,\infty)$.
\end{coro}

\bigskip

\subsection{Laplacians on graphs.}

The study of Laplacians on graphs has a long history (see, e.g., the monographs \cite{Chu,Col} and the references therein). In recent years issues such as essential self-adjointness \cite{Jor,Woj1,Web}, stochastic (in)completeness and suitable isoperimetric inequalities for infinite graphs have attracted particular attention, see e.g.\ \cite{CGY,Dod0, Dod,DK, Fuj,Jor,Kel, KL1,KL2,KP, Woj1,Woj2, Web} and the references therein. These issues can be studied in various settings. The most general setting seems to be the one introduced in \cite{KL1} (see \cite{HK,HKLW} as well), which we now recall:

Let $V$ be a countable set. Let $m$ be a measure on $V$ with full support, i.e., $m$ is a map on $V$ taking values in $(0,\infty)$.
A \emph{symmetric weighted graph over $V$} or a \emph{symmetric Markov chain on $V$} is a pair $(b,c)$ consisting of a map $c \from V \to [0,\infty)$ and
a map $b \from V\times V \to [0,\infty)$ with $b(x,x)=0$ for all $x\in V$
satisfying the following two properties:
\begin{itemize}
\item[(b1)] $b(x,y)= b(y,x)$ for all $x,y\in V$.
\item[(b2)] $\sum_{y\in V} b(x,y) <\infty$ for all $x\in V$.
\end{itemize}
Then $x,y\in V$ with $b(x,y)>0$ are called neighbors and thought to be
connected by an edge with weight $b(x,y)$. More generally, $x,y\in V$ are called connected if
there exist $x_0,x_1,\ldots,x_n \in V$ with $b(x_i, x_{i+1}) > 0$ for
$i=0,\ldots,n$ and $x_0 = x$, $x_n = y$. If any two $x,y\in V$ are connected, then $(V,b,c)$ is called connected. To $(V,b,c)$ we associate the form $Q^{\rm comp}=Q^{\rm comp}_{b,c}$ defined on the set $C_\c (V)$ of functions on $V$ with finite support by
$$
Q^{\rm comp} \from C_\c (V)\times C_\c (V) \longrightarrow [0,\infty)$$
$$Q^{\rm comp}(u,v) = \frac{1}{2}
\sum_{x,y\in V} b(x,y) (u(x) - u(y))\overline{(v(x) - v(y))} + \sum_x c(x) u(x)\overline{v(x)}.$$
Observe that the first sum is convergent by properties (b1) and~(b2); the second sum is finite.
The form $Q^{\rm comp}$ is closable in $\ell^2(V,m)$;
the closure will be denoted by $Q=Q_{b,c,m}$ and its domain by $D(Q)$.
Thus, there exists a unique selfadjoint operator $L = L_{b,c,m}$ in
$\ell^2 (V,m)$ such that
 $$ D(Q) = \mbox{Domain of definition of $L^{1/2}$}$$
and
 $$ Q(u) = \langle L^{1/2} u , L^{1/2} u\rangle_m$$
for all $u\in D(Q)$. Note that $L$ consists of essentially two parts, viz, a Laplacian type operator encoded by $b$ and a non-negative potential encoded by~$c$.

The form $Q$ is a regular Dirichlet form on $(V,m)$ (and any regular Dirichlet form on $(V,m)$ arises in this way \cite{Fuk,KL1}). Thus, the operator $-L$ generates a positivity preserving semigroup.

It is not hard to characterise when the semigroup is positivity improving.

\begin{prop} (Characterisation of positivity improvement)
Let $(V,b,c)$ be as above and let $L$ be the associated operator. Then the semigroup $(e^{-tL})_{t\ge0}$ is positivity improving if and only if $(V,b,c)$ is connected.
\end{prop}
\begin{proof}
It is clear that the semigroup cannot be positivity improving if the graph is not connected. The other implication has been shown in \cite{KL1} (see as well \cite{Web, Woj1, Dav3} for earlier results in special cases).
\end{proof}

As the underlying space is discrete, existence of kernels satisfying (K1), (K2) and (K3) is rather obvious (compare the discussion at the end of Section~\ref{Semigroups} as well). Thus, the results of the previous sections apply. We note in particular the following consequence of Theorem~\ref{kernelversion}.

\begin{coro}
Let $(V,m)$ be a discrete measure space and $(b,c)$ a graph on $V$. Let $L$ be the associated operator. Then the kernel\/ $p$ of $e^{-t L}$ satisfies
 $$ e^{t E_0} p_t (x,y) \to \Phi (x) \Phi (y) \;\: \mbox{and}\;\: \frac{\log p_t (x,y)}{t} \to -E_0 \qquad (t\to \infty)$$
for a unique $\Phi \from V \to [0,\infty)$.
\end{coro}

\textbf{Remark.}  The result of the corollary positively answers a question of Weber in \cite{Web2}.

\appendix
\section{Irreducibility and positivity preserving semigroups}\label{Irreducibility}

The crucial assumption in our results is that $e^{-t L}$ is positivity improving. It turns out that this condition is essentially equivalent to irreducibility combined with the preservation of positivity. This allows one to set up a stability theory for positivity improving semigroups. For semigroups with $E_0$ being an eigenvalue this is discussed in XIII.12 of \cite{RS}. The general case is treated, e.g., in \cite{LSV, Nagel}. For completeness reasons we shortly collect here a few items from \cite{LSV, Nagel} to which we refer for further details and results.

\bigskip

We start with the definition of irreducibility for positivity preserving semigroups. Let us emphasize that our definition is the usual one in this context. It differs from the standard definition of irreducibility in the context of selfadjoint operators by an additional assumption on invariance under multiplication by $L^\infty$ functions.

\begin{definition}
Let $L$ be a selfadjoint operator in $L^2 (X,m)$ and assume that $-L$ generates a positivity preserving semigroup $(e^{-t L})_{t\ge0}$. Then $L$ is called \emph{irreducible} if any closed subspace of $L^2 (X,m)$ which is
\begin{itemize}
\item invariant under multiplication by bounded measurable functions and
\item invariant under the semigroup,
\end{itemize}
agrees with $\{0\}$ or $L^2 (X,m)$.
\end{definition}

The following is well known. It can be found in Section~C-III.3 of \cite{Nagel} (see \cite{LSV,RS} as well).

\begin{theorem}\label{char-irreducible}
Let $L$ be a selfadjoint operator in $L^2 (X,m)$ which is bounded below. Assume that the semigroup $(e^{-t L})_{t\ge0}$ is positivity preserving. Then the following assertions are equivalent:
\begin{itemize}
\item[(i)] $e^{-t L}$ is positivity improving for one (all) $t>0$.
\item[(ii)] $(L+\alpha)^{-1} $ is positivity improving for one (all) $\alpha>-E_0(L)$.
\item[(iii)] $L$ is irreducible.
\end{itemize}
\end{theorem}

Given this theorem we can now discuss the stability of the improvement of positivity:
By the first Beurling-Deny criterion, the semigroup $(e^{-tL})_{t\ge0}$ is positivity preserving if and only if the associated form $Q$
satisfies
 $$Q (|u|, |u|) \le Q(u,u)$$
for all $u\in D(Q)$.
Now, obviously this condition is preserved when $Q$ is replaced by $Q + V$ with $V \from X \to \R$ measurable and bounded below such that the domain of definition
 $$D (Q + V) :=\set{ f\in D(Q) }{ \int_X V |f|^2 \, dm <\infty }$$
is still dense. In particular, the operator $L \dotplus V$ associated to $Q + V$ (the \emph{form sum} of $L$ and $V$) is minus the generator of a positivity preserving semigroup.

Now, we can present the following variant of Theorem in XIII.12 of \cite{RS} and its proof.


\begin{coro}\label{Stability}
Let $L$ be a selfadjoint operator in $L^2 (X,m)$ with positivity improving semigroup $(e^{-t L})_{t\ge0}$. Let $V \from X \to \R$ be measurable and bounded below satisfying the following:

\begin{itemize}

\item $D(Q + V)$ is dense in $L^2 (X,m)$.

\item There exists a sequence of bounded $V_n$ on $X$ such that $L\dotplus V - V_n$ converges to $L$ in the strong resolvent sense.

\end{itemize}
Then $L\dotplus V$ is the generator of a positivity improving semigroup.
\end{coro}
\begin{proof}
Set $L_1 := L \dotplus V$.  By denseness of $D(Q + V)$ in $L^2 (X,m)$ and the preceding discussion, the semigroup $e^{-t L_1}$ is positivity preserving. By Theorem~\ref{char-irreducible} it now suffices to show irreducibility.  Let $U$ be a closed subspace of $L^2 (X,m)$ invariant under multiplication by bounded functions and $e^{-t L_1}$ for each $t>0$. Then, by the Trotter product formula, the subspace $U$ will be invariant under $e^{-t(L_1 - V_n)}$ as well, for each $n\in \N$. By strong resolvent convergence, the subspace $U$ will be invariant under $e^{-t L}$ as well. As the latter semigroup is positivity improving, the subspace $U$ must be trivial.
\end{proof}

Note that the corollary obviously applies to bounded $V$ (see C-III.3.3 of \cite{Nagel} for this case as well).
Moreover, if $D(Q+V)$ is a core for $D(Q)$, then one can argue as in Proposition~5.8(b) of \cite{voi86} to see that the assumptions of the corollary are satisfied with $V_n = V\!\wedge n$.

\bigskip

\medskip

\textbf{Acknowledgements.}
It is our great pleasure to acknowledge fruitful and stimulating discussions with J{\'o}zef Dodziuk on the topics discussed in the paper. R. W. would like to thank Isaac Chavel and Leon Karp for introducing him to the subject and for their continued encouragment and support. His research is financially supported by FCT grant SFRH/BPD/45419/2008/ and FCT project PTDC/MAT/101007/2008. D.L. and M.K. gratefully acknowledge partial financial support from the German Research Foundation (DFG).

\end{document}